\documentclass[10pt,a4paper,reqno]{amsart}
\usepackage[latin1]{inputenc}
\usepackage[english]{babel}
\usepackage{amsmath}
\usepackage{amsfonts}
\usepackage{amssymb}
\usepackage{mathrsfs}
\usepackage{latexsym}
\usepackage{yfonts}
\usepackage{natbib}

\usepackage[margin=3cm]{geometry}

\usepackage{color}

\usepackage{graphicx,color}

\usepackage[plainpages=false,colorlinks,hyperindex,bookmarksopen,linkcolor=red,citecolor=blue,urlcolor=blue]{hyperref}

\bibpunct{[}{]}{;}{n}{,}{,}

\newtheorem{tm}{Theorem}[section]

\newtheorem{defin}{Definition}

\newtheorem{prop}{Proposition}[section]
\newtheorem{remark}{Remark}[section]

\newtheorem{example}{Example}[section]

\numberwithin{equation}{section}

\allowdisplaybreaks

\title{Delayed and rushed motions through time change}
	
	\author{Raffaela Capitanelli}
	\address{Department of Basic and Applied Sciences for Engineering\newline Sapienza University of Rome\newline via A. Scarpa 10, Rome, Italy}

	\author{Mirko D'Ovidio}
	\email[Corresponding author]{mirko.dovidio@uniroma1.it}

\begin{document}

\begin{abstract}
 We introduce a definition of delayed and rushed processes in terms of lifetimes of base processes and time-changed base processes. Then, we consider time changes given by subordinators and their inverse processes. Our analysis shows that, quite surprisingly, time-changing with inverse subordinators does not necessarily imply delay of the base process. Moreover, time-changing with subordinators does not necessarily imply rushed base process.
\end{abstract}

\date{\today}
\maketitle


\bigskip AMS-MSC : 26A33, 35R11, 60G22.

\section{Introduction}

Fractional and anomalous diffusions have a long history. The terms fractional and anomalous have been considered with different meaning and in different contexts. By fractional diffusion we mean a diffusion in a medium with fractional dimension (fractals, for instance) whereas, by anomalous diffusions, according to the most significant literature, we refer  to a motion whose mean squared displacement is proportional to a power of time. The anomalous dynamics is considered in many fields of research and many practical applications, for example in finance, physics, ecology, biology, hydrology:  the literature is huge, we mention here only few works, for example, \cite{Clark, FGIS, GGPS, GoldCox}.  Our aim is to pay exclusive attention to the probabilistic models for anomalous dynamics.

The well-known theory  of time-changed processes considers Markov and non-Markov random times. As usual we refer to subordination if the random time is Markovian. Subordinated processes are associated to subordinated semigroups (in the sense of Bochner for instance) and the generators can be represented by considering the Phillips formula (not necessarily for the generator of a Markov process). Processes obtained through non-Markovian time changes can be considered in order to solve fractional Cauchy problems. In this case, in fact, time-changed processes turn out to be non-Markovian and  the fractional operators in time are convolution operators with kernels associated with inverses of subordinators. A mathematical approach has been introduced by \cite{Baz} and further investigations have been considered by many researchers  (see, for example, \cite{Koc,Kolok,Meer, ORs}). Given a Markov process $X$ with generator $(A,D(A))$ and the inverse $L$ to a stable subordinator $H$ of order $\alpha \in (0,1)$, then $u(t,x) = \mathbf{E}_x[f(X_{L_t})]$ solves 
\begin{align*}
\partial^\alpha_t u= A u, \quad f \in D(A)
\end{align*}
where $\partial^\alpha_t$ is the Caputo fractional derivative. Since the mean squared displacement of the process $X_L$ is non linear in time (proportional to a power of time, $t^\alpha$, with $\alpha<1$) we say that $X_L$ exhibits a subdiffusive behaviour. In the literature, such a process has been also associated to a delayed process in the sense that the trajectories of $L$ may have plateaus. Since $L$ is the new clock for $X$, then the process $X_L$ is usually termed \lq\lq delayed\rq\rq .

On the other hand, if $H$ is an $\alpha$-stable subordinator, due to the fact that the trajectories of $H$ exhibit jumps, the subordinated process $X_H$ may have jumps. This should suggest that the process $X_H$ runs faster than $X$. The generator of $X_H$ is the fractional power of $-A$ given by $-(-A)^\alpha$. Here we assume that $-A$ is a non-negative definite operator. We point out that non-local operators in space may introduce non-local boundary conditions and non-local operators in time may introduce fractional initial conditions.

Recently, new fractional operators in time have been introduced in \cite{chen, toaldo}. The probabilistic representation of the solutions to the associated fractional Cauchy problems is  still obtained through time changes. This class of new fractional equations brings our attention to a new characterization of the corresponding motions. Beside the anomalous behaviour of the time-changed processes, interesting aspects are given by the comparison between the lifetimes of the base processes and those of the time-changed processes. Surprisingly, our analysis reveals that inverse processes are not necessarily related to delayed processes.

The aim of the present paper is therefore to investigate the underlying dynamics for the time changes. We introduce a precise definition of delayed and rushed processes and provide some examples which are, in some cases, counterintuitive. 

%
%

\section{Time changes and operators}
\label{secTF}

In this section, we  introduce the operators and the equations governing $X_L$ and $X_H$ for general time changes $H$ and $L$ characterized by the symbol $\Phi$.

Let $E$ be a locally compact, separable metric space   and $E_\partial= E \cup \{\partial\}$ be the one-point compactification of $E$. Denote by $\mathcal{B}(E)$ the $\sigma$-field of the Borel sets in $E$ ($\mathcal{B}_\partial$ is the $\sigma$-field in $E_\partial$).  Let $X=\{(X_t)_{t\geq 0}, (\mathbf{P}_x)_{x \in E}\}$ with infinitesimal generator $(A, D(A))$ be the continuous and symmetric Markov process on $(E, \mathcal{B}(E))$ with transition function $p(t, x, B)$ on $[0, \infty) \times E \times \mathcal{B}(E)$. The point $\partial$ is the cemetery point for $X$ and a function $f$ on $E$ can be extended to $E_\partial$ by setting $f(\partial)=0$. The associated semigroup is uniquely defined by
\begin{align*}
P_t f(x):= \int_E p(t,x,dy) f(y) = \mathbf{E}_x[f(X_t)], \quad f \in C_\infty(E)
\end{align*}
where $\mathbf{E}_x$ denote the mean value with respect to $\mathbf{P}_x$ with $X_0=x \in E$ and $C_\infty$ is the set of continuous function $C(E)$ on $E$ such that $f(x)\to 0$ as $x\to \partial$. Let $\mathcal{E}(u,v)= (\sqrt{-A}u, \sqrt{-A}v)$ with domain $D(\mathcal{E})= D(\sqrt{-A})$ be the Dirichlet form associated with (the non-positive definite, self-adjoint operator) $A$. Then $X$ is equivalent to an $m$-symmetric Hunt process whose Dirichlet form $(\mathcal{E}, D(\mathcal{E}))$ on $L^2(E,m)$ is local (see  the books \cite{Chen-book, Fuk-book}). Without restrictions we assume that the form is regular. Then, $X$ is a continuous strong Markov process (we say Feller diffusion).\\

We now introduce the Bernstein function 
\begin{equation}
\Phi(\lambda) = \int_0^\infty \left( 1 - e^{ - \lambda z} \right) \Pi(dz), \quad \lambda \geq 0 
\label{LevKinFormula}
\end{equation} 
where $\Pi$ on $(0, \infty)$ with $\int_0^\infty (1 \wedge z) \Pi(dz) < \infty$ is the associated L\'evy measure.  We also recall that
\begin{align}
\label{tailSymb}
\frac{\Phi(\lambda)}{\lambda} = \int_0^\infty e^{-\lambda z} \overline{\Pi}(z)dz, \qquad \overline{\Pi}(z) = \Pi((z, \infty))
\end{align}
and $\overline{\Pi}$ is the so called \emph{tail of the L\'evy measure}. For details, see the book  \cite{BerBook}. The symbol $\Phi$ can be associated with the Laplace exponent of a subordinator $H$, that is 
$$\mathbf{E}_0[\exp( - \lambda H_t)] = \exp(- t \Phi(\lambda)).$$ 
We introduce the inverse process 
$$L_t = \inf \{s \geq 0\,:\, H_s >t \}$$ 
and define the time-changed process 
$$X^{L}_t :=X_{L_t}, \quad t \geq 0$$  
as the composition $X \circ L$ and the time-changed process
$$X^H_t := X_{H_t}, \quad t>0$$ 
as the composition $X \circ H$. We do not consider step-processes with $\Pi((0, \infty)) < \infty$ and therefore we focus only on strictly increasing subordinators with infinite measures. Thus, the inverse process $L$ turns out to be a continuous process with non-decreasing paths. By definition, we also can write
\begin{align}
\label{relationHL}
\mathbf{P}_0(H_t < s) = \mathbf{P}_0(L_s>t), \quad s,t>0.
\end{align}

Let $M>0$ and $w\geq 0$. Let $\mathcal{M}_w$ be the set of (piecewise) continuous function on $[0, \infty)$ of exponential order $w$ such that $|u(t)| \leq M e^{wt}$. Denote by $\widetilde{u}$ the Laplace transform of $u$. Then, we define the operator $\mathfrak{D}^\Phi_t : \mathcal{M}_w \mapsto \mathcal{M}_w$ such that
\begin{align*}
\int_0^\infty e^{-\lambda t} \mathfrak{D}^\Phi_t u(t)\, dt = \Phi(\lambda) \widetilde{u}(\lambda) - \frac{\Phi(\lambda)}{\lambda} u(0), \quad \lambda > w
\end{align*}
where $\Phi$ is given in \eqref{LevKinFormula}. Since $u$ is exponentially bounded, the integral $\widetilde{u}$ is absolutely convergent for $\lambda>w$.  By Lerch's theorem the inverse Laplace transforms $u$ and $\mathfrak{D}^\Phi_tu$ are uniquely defined. We note that
\begin{align}
\label{PhiConv}
\Phi(\lambda) \widetilde{u}(\lambda) - \frac{\Phi(\lambda)}{\lambda} u(0) = & \left( \lambda \widetilde{u}(\lambda) - u(0) \right) \frac{\Phi(\lambda)}{\lambda}
\end{align}
and thus, $\mathfrak{D}^\Phi_t$ can be written as a convolution involving the ordinary derivative and the inverse transform of \eqref{tailSymb} iff $u \in \mathcal{M}_w \cap C([0, \infty), \mathbb{R}_+)$ and $u^\prime \in \mathcal{M}_w$. That is,
\begin{align*}
\mathfrak{D}^\Phi_t u(t) = \int_0^t u^\prime(s) \overline{\Pi}(t-s)ds.
\end{align*}
By Young's inequality we also observe that
\begin{align}
\label{YoungSymb}
\int_0^\infty |\mathfrak{D}^\Phi_t u |^p dt \leq \left( \int_0^\infty |u^\prime |^p dt \right) \left( \lim_{\lambda \downarrow 0} \frac{\Phi(\lambda)}{\lambda} \right)^p, \qquad p \in [1, \infty)
\end{align}
where $\lim_{\lambda \downarrow 0} \Phi(\lambda) /\lambda$ is finite only in some cases (\cite{CapDovVIA}) .
We notice that when $\Phi(\lambda)=\lambda$ (that is we deal with the ordinary derivative) we have that $H_t = t$ and $L_t=t$ a.s. and in \eqref{YoungSymb} the equality holds. For explicit representation of the operator $\mathfrak{D}^\Phi_t$ see also the recent works \cite{chen, toaldo}. 

\begin{remark}
\label{remark:telegr}
We notice that for $\Phi(\lambda)=\lambda^\beta$, the symbol of a stable subordinator, the operator $\mathfrak{D}^\Phi_t$ becomes the Caputo fractional derivative
\begin{align*}
\mathfrak{D}^\Phi_t u(t) = \frac{1}{\Gamma(1-\beta)} \int_0^t \frac{u^\prime(s)}{(t-s)^\beta}ds
\end{align*}
with $u^\prime(s)=du/ds$. 
A further example is given by the symbol $\Phi(\lambda)=\lambda^{2\beta} + \lambda^\beta$ for $\beta \in (0, 1/2)$, that is, $\mathfrak{D}^\Phi_t$ becomes the telegraph fractional operator (\cite{DOTtelegraph}).

The case $\beta=1$ is an interesting case of exponential velocity correlation. If $\mathcal{T}_t$ is the telegraph process, then $
\mathcal{T}_t = \int_0^t V_s\, ds \quad \textrm{with} \quad V_s = V_0  (-1)^{N_s}$ where $N_s$, $s\geq 0$ is an homogeneous Poisson process of rate $\lambda$ independent from $V_0$ with $\mathbf{P}(V_0=+c)=\mathbf{P}(V_0=-c)$, $c>0$. The process on the real line is performed by a random motion with finite speed $c$. The second moment is proportional to $t\,c^2/\lambda$ for large $t$ (\cite{orsT}). For the first passage time $\tau_a$ through the level $a$ we have that $\mathbf{E}_x[\tau_a] = \infty$ (see \cite{Foong, orsT}). This well accords with the fact that $\mathcal{T}_t$ is a process with finite velocity, the process spends an infinite (mean) amount of time in a set.
\end{remark}

We now introduce the Phillips representation (\cite{Phillips})
\begin{align*}
-\Phi(-A) f(x) = \int_0^\infty \big( P_z f(x) -  f(x) \big) \Pi(dz) 
\end{align*}
where $\Phi$ has been given in \eqref{LevKinFormula}. Let $\Psi$ be the Fourier multiplier of $A$. Then the Fourier symbol of the semigroup $P_t = e^{tA}$ is written as $\widehat{P_t} = e^{-t \Psi}$. For a well-defined function $f$, from \eqref{LevKinFormula} we have that
\begin{align*}
\int_{\mathbb{R}} e^{i\xi x} \left( -\Phi(-A) f(x) \right) dx = \left( \int_0^\infty \left( e^{-z \Psi(\xi)} -1 \right) \Pi(dz) \right) \widehat{f}(\xi)= -\Phi \circ \Psi(\xi)\, \widehat{f}(\xi).
\end{align*}
We also recall that composition of Bernstein functions is a Bernstein function.

\begin{remark}
We note that if $H$ is the stable subordinator with symbol  $\Phi(\xi)=\xi^\alpha$ and  $X$ is the Brownian motion with $\Psi(\xi)= \xi^2$,  the process $X^H$ is the symmetric stable process with symbol $|\xi |^{2\alpha}$ driven by the fractional Laplacian
\begin{align}
\label{fracLaplacianS}
-\Phi(- A) = -(-\Delta)^{\alpha}, \quad \alpha\in (0,1).
\end{align}
\end{remark}

%

\section{Time-changed process $X^L_t$}
\label{secSF}

In this section we consider the time fractional equation
\begin{align}
\label{time-frac-problem}
\mathfrak{D}^\Phi_t u = A u, \qquad u_0=f \in D(A).
\end{align}

The probabilistic representation of the solution to \eqref{time-frac-problem} is written in terms of the time-changed process $X^{L}$, that is 
\begin{align}
\label{sol-time-frac-problem}
u(t,x) = \mathbf{E}_x[f(X^{L}_t)]= \mathbf{E}_x[f({^*X}^{L}_t), t < \zeta^{L}]
\end{align}
where $\zeta^{L}$ is the lifetime of $X^{L}$, the part process of ${^*X}^{L}$ on $E$. The fact that $L$ is continuous implies that 
\begin{align*}
\mathbf{E}_x[f({^*X}^{L}_t), t < \zeta^{L}] = \mathbf{E}_x[f({^*X}_{L_t}), \, L_t < \zeta]
\end{align*}
(see for example \cite[Corollary 3.2]{Meer} or \cite[Section 2]{Svond}). Moreover, we have that
\begin{align}
\label{semigPhi-L-continuity}
\mathbf{E}_x[f({^*X}^{L}_t), t < \zeta^{L}] =  \int_0^\infty P_s f(x) \mathbf{P}_0(L_t \in ds) :=P^{L}_t f(x) 
\end{align}
where $P_tf(x) = \mathbf{E}_x [f(X_t)] = \mathbf{E}_x [f({^*X}_t), t < \zeta]$ and $\zeta$ is the lifetime of $X$, the part process of ${^*X}$ on $E$. Then, for the  time-changed processes we have that
\begin{align}
\label{semigPhi-1}
P^{L}_t \mathbf{1}_E (x) = \mathbf{P}_x(t < \zeta^{L}) = \mathbf{P}_x(L_t < \zeta) = \int_0^\infty \mathbf{P}_x(s < \zeta) \mathbf{P}_0(L_t \in ds).
\end{align}
We notice that $P_t^{L}$ is not a semigroup (indeed the random time is not Markovian). We also introduce the following $\lambda$-potentials ($\lambda>0$)
\begin{align*}
R_\lambda f(x) = \mathbf{E}_x \left[ \int_0^\infty e^{-\lambda t} f(X_t)dt \right], \qquad R^{L}_\lambda f(x) = \mathbf{E}_x \left[ \int_0^\infty e^{-\lambda t} f(X^{L}_t)dt \right].
\end{align*}

The following result concerns the relation between the lifetime of the time-changed process  $X^L$ with the lifetime of the random time $H$.
\begin{tm}
\label{meanLzeta}
Let $\Phi$ be the Bernstein function in \eqref{LevKinFormula}. Let $H$ be the subordinator with symbol $\Phi$. We have
\begin{align}
\mathbf{E}_x[\zeta^L] = \mathbf{E}_x[H_\zeta], \quad x \in E.
\end{align}
\end{tm}
\begin{proof}
From \eqref{semigPhi-1} 
\begin{align}
\label{NoSemig}
\mathbf{E}_x[\zeta^L] 
= & \lim_{\lambda \to 0} R^L_\lambda \mathbf{1}_E(x) = \int_0^\infty  P^L_t \mathbf{1}_E(x)\, dt  = \int_0^\infty \mathbf{E}_x[L_t < \zeta] \, dt.
\end{align}
Since
\begin{align*}
 \mathbf{E}_x[L_t < \zeta]  
 = & \int_0^\infty \mathbf{P}_x(\zeta >s) \frac{d}{d s} \mathbf{P}_0(L_t < s) ds \\
 = & \mathbf{P}_x(\zeta>s) \mathbf{P}_0(L_t < s) \bigg|_{s=0}^{s=\infty} + \int_0^\infty \mathbf{P}_x(\zeta \in ds) \mathbf{P}_0(L_t < s) \\
 = & \int_0^\infty \mathbf{P}_x(\zeta \in ds) \mathbf{P}_0(L_t < s),
\end{align*}
from the relation \eqref{relationHL}, 
 \begin{align*}
 \mathbf{E}_x[L_t < \zeta] =  \int_0^\infty \mathbf{P}_x(\zeta \in ds) \mathbf{P}_0(H_s >t)= \mathbf{E}_x[ t < H_\zeta].
 \end{align*}
Then,
\begin{align*}
\int_0^\infty  \mathbf{E}_x[L_t < \zeta] \, dt = \int_0^\infty \mathbf{P}_x(H_\zeta >t)dt = \mathbf{E}_{x}[H_\zeta].
\end{align*}
\end{proof}

Now we recall the following result given in \cite{CapDovVIA} (see also \cite{chen}) that relates the lifetime of the time changed process $X^L$ with the lifetime of the base process $X$.
\begin{tm}
\label{B}
Let $\Phi$ be the Bernstein function in \eqref{LevKinFormula}. Let $L$ be the inverse of $H$ with symbol $\Phi$. We have 
  \begin{equation}
\mathbf{E}_x[\zeta^{L}]=
\lim_{\lambda \downarrow 0} \frac{\Phi(\lambda)}{\lambda} \mathbf{E}_x[\zeta].\end{equation}
\end{tm}
\begin{proof}
From \eqref{semigPhi-1} and \eqref{relationHL}, for the $\lambda$-potential above we obtain
\begin{align}
R^{L}_\lambda \mathbf{1}_{E}(x) = & \int_0^\infty \mathbf{P}_x(s < \zeta) \int_0^\infty e^{-\lambda t} \frac{d}{ds} \mathbf{P}_0(H_s >t) dt\, ds \notag \\
= & \frac{1}{\lambda} \int_0^\infty \mathbf{P}_x(s < \zeta) \frac{d}{ds}\left(1 - \mathbf{E}_0[e^{-\lambda H_s}] \right)ds 
= \frac{1}{\lambda} \mathbf{E}_x[1 - e^{-\zeta \Phi(\lambda)}].\label{potential-equality}
\end{align}
By considering \eqref{potential-equality} we have that, as $\lambda \to 0$,
\begin{align}
R^{L}_\lambda \mathbf{1}_{E}(x) = \frac{1}{\lambda} \mathbf{E}_x[1 - e^{-\zeta \Phi(\lambda)}] \to \left( \lim_{\lambda \downarrow 0} \frac{\Phi(\lambda)}{\lambda} \right) \mathbf{E}_x[\zeta] = \mathbf{E}_x[\zeta^{L}]
\end{align}
which is the mean lifetime of the time-changed process $X^{\Phi}$.
\end{proof}

\section{Time-changed process $X^H_t$}
\label{secTSF}

For the process $X$ with generator $(A, D(A))$ introduced before and the independent subordinator $H$ with symbol \eqref{LevKinFormula}, the process $X^H_t=X_{H_t}$, $t\geq 0$ (the composition $X \circ H$) can be considered in order to solve the equation
\begin{align}
\label{space-frac-problem} 
\frac{\partial u}{\partial t} = - \Phi( - A) u, \quad u_0=f \in D(\Phi( - A)) \subset D(A).
\end{align} 
More precisely, the probabilistic representation of the solutions to \eqref{space-frac-problem} is given by
\begin{align*}
P^{H}_t f(x):=\mathbf{E}_x [f(X^{H}_t)]= \mathbf{E}_x [f(X_{H_t})] 
\end{align*}
and (see \cite[Section 2]{Svond} or \cite[Section 3]{SVrel} for a detailed discussion)
\begin{align*}
\mathbf{E}_x [f(X^{H}_t)] = \mathbf{E}_x [f({^*X}^{H}_t), t < \zeta^{H}]
\end{align*}
where $\zeta^{H}$ is the lifetime of $X^{H}$, that is the part process of ${^*X}^{H}$ on $E$. We recall that the multiplicative functional $M_t = \mathbf{1}_{(t<\zeta^H)}$ characterizes uniquely the semigroup $P^H_t$ (see \cite[pag. 100, Proposition 1.9]{BluGet68}). We also introduce the following $\lambda$-potential ($\lambda>0$)
\begin{align*}
R^{H}_\lambda f(x) = \mathbf{E}_x \left[ \int_0^\infty e^{-\lambda t} f(X^{H}_t)dt \right].
\end{align*}
We notice that if $H$ is a stable subordinator with symbol $\Phi(\lambda)=\lambda^\alpha$ and $X$ is a Brownian motion, then $-\Phi(-A)$ is the fractional Laplacian (see, for example, \cite{DovJMA, MK} for details and connections between fractional Laplacians, subordinators and random walks continuous in time and space).

The next theorem relates the lifetime of the time-changed process $X^H$ with the lifetime of the random time $L$.

\begin{tm}
\label{meanHzeta}
Let $\Phi$ be the Bernstein function \eqref{LevKinFormula}. Let $L$ be the inverse to a subordinator $H$.  We have that 
\begin{align}
\mathbf{E}_x[\zeta^H] = \mathbf{E}_x[L_\zeta], \quad x \in E.
\end{align}
\end{tm}
\begin{proof}
We observe that, as $\lambda\to 0$,
\begin{align*}
R^H_\lambda \mathbf{1}_E(x) \to& \int_0^\infty \int_0^\infty \mathbf{P}_x(X_s \in E) \mathbf{P}_0(H_t \in ds)\, dt\\
= & \int_0^\infty \int_0^\infty \mathbf{P}_x(\zeta >s) \mathbf{P}_0(H_t \in ds)\, dt
=  \int_0^\infty \mathbf{P}_x(\zeta >H_t)\, dt
\end{align*}
where
\begin{align*}
\mathbf{P}_x(\zeta >H_t) 
= & \int_0^\infty \mathbf{P}_x(\zeta >s) \frac{d}{d s} \mathbf{P}_0(H_t < s) ds\\ 
= & \int_0^\infty \mathbf{P}_x(\zeta \in ds) \mathbf{P}_0(L_s >t)
=  \mathbf{P}_x(L_\zeta >t)
\end{align*}
and therefore
\begin{align*}
R^H_\lambda \mathbf{1}_E(x) \to &  \int_0^\infty \mathbf{P}_x(L_\zeta >t)dt = \mathbf{E}_x[L_\zeta].
\end{align*}
Since
\begin{align*}
R^H_\lambda \mathbf{1}_E(x) \to \mathbf{E}_x[\zeta^H], \quad \textrm{as} \quad \lambda\to 0
\end{align*}
we obtain the result.
\end{proof}

Let us focus on the case $\Phi(\lambda)=\lambda^\alpha$. We first recall that $L$, the inverse to a subordinator $H$, can be characterized as follows
\begin{align}
\label{Lap-mittag-L-F}
\mathbf{E}_0 [\exp (-\lambda L_t)] = E_\alpha(-\lambda t^\alpha), \quad \lambda > 0
\end{align}
where, for $\alpha, \gamma>0$,
\begin{align*}
E_{\alpha, \gamma}( -z) = \sum_{k \geq 0} \frac{(-z)^k}{\Gamma(\alpha k + \gamma)}, \quad z\geq 0
\end{align*}
is the generalized Mittag-Leffler function. For $\gamma=1$,  $E_\alpha = E_{\alpha, 1}$ which is the so called Mittag-Leffler function. 

The next theorem gives a precise characterization of the lifetime of $X^H$ in terms of the symbol $\Phi$.

\begin{tm}
\label{A}
Let $\Phi(\lambda) = \lambda^\alpha$. Let $H$ be the subordinator with symbol $\Phi$. Then
\begin{align*}
\mathbf{E}_x[\zeta^{H}] = & \frac{1}{\Gamma(\alpha +1)} \mathbf{E}_x \left[ \Phi(\zeta) \right] =  \lim_{\lambda \downarrow 0} \frac{\mathbf{E}_x \left[ \Phi(\zeta) \right]}{\mathfrak{D}^\Phi_\lambda \Phi(\lambda)}
\end{align*}
and 
\begin{align*}
\lim_{s\to \infty} \Phi(s) \mathbf{P}_x(\zeta >s) =0.
\end{align*}
Moreover, if 
$$\mathbf{P}_x(\zeta > s) \leq  e^{-\omega s}, \quad \textrm{for some } \; \omega>0$$
then
\begin{align*}
\mathbf{E}_x[\zeta^{H}] \leq \frac{1}{\gamma^\alpha} \mathbf{E}_x[ e^{\gamma \zeta}] , \quad \gamma < \omega
\end{align*}
\end{tm}
\begin{proof}
We have that
\begin{align}
P^{H}_t \mathbf{1}_{E}(x) = \int_0^\infty \mathbf{P}_x(s < \zeta) \mathbf{P}_0(H_t \in ds)
\label{semigH-1}
\end{align}
is a semigroup (the random time is Markovian) with the associated  resolvent 
\begin{align}
R^{H}_\lambda \mathbf{1}_E(x) = & \mathbf{E}_x \left[\int_0^\infty e^{-\lambda t} \mathbf{1}_E(X^H_t) dt \right], \quad \lambda>0.
\end{align}
Since $\mathbf{P}_0(H_t < x) = \int_0^x h(t, s)ds$ where
\begin{align*}
\int_0^\infty e^{-\lambda t} h(t,s) dt = & \int_0^\infty e^{-\lambda t} \frac{d}{d s} \mathbf{P}_0(H_t < s)dt
=  \frac{d}{d s} \int_0^\infty e^{-\lambda t} \mathbf{P}_0(L_s > t)dt \\
= & \frac{d}{d s} \left( \frac{1}{\lambda} - \frac{1}{\lambda} \int_0^\infty e^{-\lambda t} \mathbf{P}_0(L_s \in dt) \right)
=  \frac{d}{d s} \left( \frac{1}{\lambda} - \frac{1}{\lambda} E_\alpha(- \lambda s^\alpha) \right)
\end{align*}
and 
\begin{align}
\frac{d}{d s} E_\alpha(-\lambda s^\alpha) = - \lambda s^{\alpha -1} E_{\alpha, \alpha}(- \lambda s^\alpha),
\label{derivML}
\end{align}
we obtain 
\begin{align}
R^{H}_\lambda \mathbf{1}_E(x) = & \mathbf{E}_x \left[\int_0^\infty e^{-\lambda t} \mathbf{1}_E(X_{H_t}) dt \right]
= \int_0^\infty e^{-\lambda t} \int_0^\infty \mathbf{P}_x(X_{s} \in E) \mathbf{P}_0 (H_t \in ds)\, dt \notag \\
= &  \int_0^\infty \mathbf{P}_x(s < \zeta) E_{\alpha,\alpha}(- \lambda s^\alpha) \, s^{\alpha-1} ds \label{potentialML}\\
= & \frac{1}{\lambda} \mathbf{E}_x [1 - E_\alpha(-\lambda (\zeta)^\alpha)]. \label{potentialML-2}
\end{align}
Thus, from \eqref{potentialML-2} we have that, as $\lambda\to 0$,
\begin{align}
\label{limLAMBDA1}
R^{H}_\lambda \mathbf{1}_E(x) \to \mathbf{E}_x [\zeta^{H}] 
\end{align}
where
\begin{align*}
\mathbf{E}_x [\zeta^{H}] = \frac{1}{\Gamma(\alpha +1)} \mathbf{E}_x\left[ (\zeta)^\alpha \right]. 
\end{align*}
Moreover, due to the fact that $\Phi(0)=0$, the Riemann-Liouville and the Caputo derivatives of $\Phi$ coincides. In particular, 
$$\mathfrak{D}^\Phi_\lambda \Phi(\lambda) = \frac{\Gamma(1+ \alpha)}{\Gamma(1+\alpha - \alpha)} \lambda^{\alpha - \alpha}={\Gamma(\alpha +1)} .$$ 
Passing to the limit in \eqref{potentialML} we have that
\begin{align*}
\lim_{\lambda \to 0} R^{H}_\lambda \mathbf{1}_E(x) = & \frac{1}{\Gamma(\alpha +1)} \left( s^\alpha \mathbf{P}_x(\zeta >s) \bigg|_{s=0}^{s=\infty} + \mathbf{E}_x[(\zeta)^\alpha] \right)
\end{align*}
and, from \eqref{limLAMBDA1}, it holds that
\begin{align*}
\lim_{s\to \infty} s^{\alpha} \mathbf{P}_x(\zeta >s) =0
\end{align*}
and this concludes the first part of the proof.\\

Since we have that
\begin{align}
R^{H}_\lambda \mathbf{1}_E(x) = & \int_0^\infty \frac{s^\alpha}{s} E_{\alpha, \alpha} (-\lambda s^\alpha) \mathbf{P}_x(\zeta >s) ds \label{firstRpsi}\\ 
\leq & \int_0^\infty \frac{s^\alpha}{s} E_{\alpha, \alpha} ( - \lambda s^\alpha) \frac{\mathbf{E}_x[g(\zeta)]}{g(s)} ds, \qquad \forall\, \lambda \geq 0 \notag
\end{align}
for $g$ increasing and positive, for $g(s)=e^{\gamma s}$, under the assumption that 
\begin{align*}
\mathbf{P}_x(\zeta > s) \leq e^{-\omega s}, \quad \textrm{for some } \; \omega>0,
\end{align*}
we obtain that, as $\lambda \to 0$,
\begin{align*}
R^{H}_\lambda \mathbf{1}_E(x) \leq \frac{1}{\gamma^\alpha + \lambda} \mathbf{E}_x[ e^{\gamma \zeta}] \to \frac{1}{\gamma^\alpha} \mathbf{E}_x[ e^{\gamma \zeta}] , \quad \gamma < \omega
\end{align*}
where we used (\cite{HMS})
\begin{align*}
\int_0^\infty e^{-\gamma s} s^{\alpha-1} E_{\alpha, \alpha}(-\lambda s^\alpha) ds = \frac{1}{\gamma^\alpha + \lambda}.
\end{align*}
\end{proof}

Let us consider now the Dirichlet boundary condition on $\partial D$ for which the Brownian motion $X$ on $D$ has lifetime $\zeta=\tau_D$, the first exit time from $D$. A killed subordinate Brownian motion on the bounded domain $D\subset \mathbb{R}^d$ is associated with the Dirichlet form on $L^2(D, dx)$ given by
\begin{align*}
\int_D \int_D \big( u(x) - u(y) \big)^2 J(x,y)dx\, dy + \int_D |u(x)|^2\, \kappa^*(x)\, dx, \quad u \in H^\alpha_0(D)
\end{align*} 
with the jumping measure $J(x, y)=C(d,-2\alpha)/2 |x-y|^{-(d+2\alpha)}$ and the killing measure 
\begin{align*}
\kappa^*(x) = C(d,-2\alpha) \int_{D^c} \frac{dy}{|x-y|^{d+2\alpha}}
\end{align*}
($C(d,-2\alpha)$ is a constant depending on $d$ and $\alpha \in (0,1)$). In \cite{Svond} the authors provide an interesting discussion comparing the killed subordinate Brownian motion (with killing measure $\kappa^*$) with the subordinate killed Brownian motion (in our notation $X^H$, $H$ is a stable subordinator of order $\alpha \in (0,1)$). In particular, they show that such processes have comparable killing measures. Denote by $\kappa(\cdot)$ the killing measure of $X^H$ where $H$ has symbol $\Phi(\lambda)=\lambda^\alpha$ and $X$ is a killed Brownian motion on $D \subset \mathbb{R}^d$. Than, we have that (\cite[Lemma 3.1]{Svond})
\begin{align*}
\kappa(x) = \frac{1}{\Gamma(-\alpha + 1)} \mathbf{E}_x[(\zeta)^{-\alpha}], \quad x \in \mathbb{R}^d.
\end{align*}  
We notice that for the subordinate killed Brownian motion $X^H$ on $D \subset \mathbb{R}^d$ we have that  (by applying Theorem \ref{A})
\begin{align*}
\mathbf{E}_x[\zeta^H] = \frac{1}{\Gamma(\alpha + 1)} \mathbf{E}_x[(\zeta)^\alpha], \quad x \in \mathbb{R}^d
\end{align*}

We think that the previous theorem can be generalized  for an arbitrary $\Phi$. In the following proposition we generalize the Lemma 3.1 in \cite{Svond}.
\begin{prop}
Let $X$ be the $d-$dimensional killed Brownian motion on $D$. Let $H$ be the subordinator with symbol $\Phi$ defined in \eqref{LevKinFormula}. Let $\Phi$ be regularly varying at $+\infty$. If 
\begin{align}
\label{condZERO}
\lim_{s\to 0} \Phi(1/s) \mathbf{P}_x(\zeta \leq s ) = 0 
\end{align}
then, the killing measure of $X^H$ is written as
\begin{align}
\label{killingMeas}
\kappa(x) = \mathbf{E}_x[\overline{\Pi}(\zeta)], \quad x \in \mathbb{R}^d
\end{align}
where $\overline{\Pi}$ has been defined in \eqref{tailSymb}.
\end{prop}
\begin{proof}
We follows the proof of Lemma 3.1 given in \cite{Svond} starting for the fact that
\begin{align*}
d \overline{\Pi}(s) = d\Pi((s,\infty)) = - \Pi(s) ds
\end{align*}
and
\begin{align*}
\mathbf{E}_x[\overline{\Pi}(\zeta)] = \int_0^\infty \overline{\Pi}(s) \,d\mathbf{P}_x(\zeta \leq s) = - \int_0^\infty \mathbf{P}_x(\zeta \leq s)\, d\overline{\Pi}(s)
\end{align*}
where, in the integration by parts, we used \cite[Proposition 1.5]{BerBook} and \eqref{condZERO}.
\end{proof}

We now provide a further result obtained in a fruitful discussion with Prof. Zoran Vondra\v{c}ek about the previous proposition. He noticed that for every subordinate semigroup it holds that
\begin{align*}
\kappa(x) = \int_0^\infty (1-P_t \mathbf{1}(x)) \Pi(dt).
\end{align*}
Thus, by exploiting the fact that $1-P_t\mathbf{1}(x) = \mathbf{P}_x(\zeta  \leq t)$, 
\begin{align*}
\kappa(x) = \int_0^\infty \mathbf{P}_x(\zeta \leq t) \, \Pi(dt) = \int_0^\infty \mathbf{E}_x[1_{(\zeta > t)}] \Pi(dt) 
\end{align*}
and, by Fubini's theorem,
\begin{align*}
\kappa(x) = \mathbf{E}_x \left[ \int_0^\infty 1_{(\zeta > t)} \Pi(dt)\right] = \mathbf{E}_x [\Pi((\zeta, \infty))] = \mathbf{E}_x[\overline{\Pi}(\zeta)],
\end{align*}
that is, we obtain \eqref{killingMeas} without any conditions about $\Phi$. Then, we state the following result.
\begin{prop}
Let $X$ be the $d-$dimensional killed Brownian motion on $D$. Let $H$ be the subordinator with symbol $\Phi$ defined in \eqref{LevKinFormula}. Let $\Phi$ be regularly varying at $+\infty$. Then 
\begin{align*}
\lim_{s\to 0} \Phi(1/s) \mathbf{P}_x(\zeta \leq s ) = 0.
\end{align*}
\end{prop}
\begin{proof}
From the previous proof we immediately see that \eqref{killingMeas} implies \eqref{condZERO}. The observation by Prof. Vondra\v{c}ek says that
\begin{align*}
\lim_{s\to 0} \overline{\Pi}(s) \mathbf{P}_x(\zeta \leq s ) = 0
\end{align*}
for any L\'{e}vy measure $\Pi$. Thus, from \cite[Proposition 1.5]{BerBook} we obtain a characterization in terms of the symbol $\Phi$.
\end{proof}

\begin{remark}
We observe that the results obtained in Section \ref{secTF}, Section \ref{secSF} and Section \ref{secTSF} still hold for the wide class of continuous Markov processes.
\end{remark}

\section{Delayed and Rushed  anomalous diffusions}

Normal and anomalous diffusions differ for the rate at which velocity correlation decreases to zero. Normal diffusion occurs if the velocity correlation decreases rapidly whereas anomalous diffusion is concerned with processes moving coherently for long times with no (or not frequent) changes of direction. This can be associated with the tail behaviour of the autocorrelation function: if the correlation function decays exponentially, then there is normal diffusion, whereas if the correlation function decays algebraically, then there is the possibility of anomalous diffusion. Thus, anomalous diffusion should be related with non-Markovian dynamics. However, the definition commonly used is based on the second moment (or mean square displacement) proportional to a power of time, say $t^\gamma$, for which we say that the process exhibits a subdiffusive behaviour if $\gamma<1$ (long tailed distribution in time) or a superdiffusive behaviour if $\gamma>1$ (long tailed distribution in space). 
Some other characterizations have been also considered, for example ultra-slow and strong anomalous diffusions. The ultra-slow diffusive behaviour is described by a mean square displacement proportional to $(\log t )^\gamma$ (Sinai-like diffusion). Strong anomalous diffusion has been investigated in the interesting paper \cite{CastVulp}.

Throughout 
we dealt with anomalous diffusion just considering the anomalous behaviour given by the second moment (or the mean square displacement). \\

 Anomalous diffusions have been also considered in terms of the second moment proportional to $t^{2/d_w}$ where $d_w$ is the random walk dimension. This scaling behaviour occurs in a variety of circumstances and have been observed with exponents $2/d_w$ smaller and larger than $1$, that is the ordinary Brownian motion on $\mathbb{R}^d$. For $d_w>2$ we refer to the process as subdiffusive (charge carrier transport in amorphous semiconductors, the motion of a bead in a polymer network, diffusion on fractals, see for example \cite{Mosco}). For $d_w < 2$ we refer to the process as superdiffusive (diffusion of two particles in turbulent flows, diffusion of tracer particles in vortex arrays in a rotating flow, layered velocity fields). \\

We would like to recall that the non-linear diffusion equation
$$\partial_t u = \Delta u^m = \textsl{div}(m \,u^{m-1} \nabla u)$$
gives rise to a diffusion called normal (if $m=1$), slow (if $m>1$), fast (if $m<1$), ultra-fast (if $m<0$) depending on the diffusivity $D(u) = m \, u^{m-1}$ (see, for example, \cite{Vasq}). In the present paper, we consider slow/fast diffusion meaning subdiffusive/superdiffusive behaviour in the anomalous motion given by the mean square displacement.\\

In this scenario, the main aspect to be analysed seems to be how the anomalous behaviour modifies first passage times. More precisely, how does the time change modify first passage times? 

We start with the following definition given in \cite[formulas (1.1) and (1.2)]{BCM} for the Brownian motion.
\begin{defin}
\label{def1}
Let $E \subset \mathbb{R}^d$ be an open connected set with finite volume. Let $B \subset E$ be a closed ball with non-zero radius. Let $X$ be a reflected Brownian motion on $\overline{E}$ and denote by $T_B= \inf\{t\geq 0\,:\, X_t \in B\}$ the first hitting time of $B$ by $X$. We say that $E$ is a \emph{trap domain} for $X$ if 
\begin{align*}
\sup_{x \in E} \mathbf{E}_x[T_B] = \infty.
\end{align*}
Otherwise, we say that $E$ is a \emph{non-trap domain} for $X$.
\end{defin}
In Definition \ref{def1}, the random time $T_B$ plays the role of lifetime for the Brownian motion on $\bar{E} \setminus \bar{B}$ reflected on $\partial E \setminus \partial B$ and killed on $\partial B$. On the other hand, $\mathbf{E}_x[T_B]< \infty$ for any $x \in E$ (\cite[Lemma 3.2]{BCM}). This aspect suggests to pay particular attention on the boundary. A process  may have an infinite lifetime depending on the regularity of the boundary $\partial E$ where it can be trapped.\\ 

Further on we denote by $\zeta$ (possibly with some superscript) the lifetime of a process $X$, that is for the process $X_t$ in $E$ with $X_0=x \in E$ (denote by $E^c$ the complement set of $E$),
\begin{align*}
\zeta := \inf\{t >0\,:\, X_t \in E^c\}.
\end{align*}
Let $T$ be a random time and denote by $X^T := X \circ T$ the process $X$ time-changed by $T$. It is well-known that $X^T$ is Markovian only for a Markovian time change $T$, otherwise from the Markov process $X$ we obtain a non-Markov process $X^T$. 

We introduce the following characterization of $X$ introduced in Section \ref{secTF} and the time-changed process $X^T$ in terms of the corresponding lifetimes. Denote by $\zeta^T$ the lifetime of $X^T$.
\begin{defin}
\label{defDelRus}
Let $E \subset \mathbb{R}^d$. 
\begin{itemize}
\item[-] We say that $X$ is delayed by $T$ if $\mathbf{E}_x[\zeta^T] > \mathbf{E}_x[\zeta]$, $\forall \, x \in E$.
\item[-] We say that $X$ is rushed by $T$ if $\mathbf{E}_x[\zeta^T] < \mathbf{E}_x[\zeta]$, $\forall \, x \in E$.
\end{itemize}
Otherwise, we say that $X$ runs with its velocity.
\end{defin}

\begin{remark}
Let $X$ be a Brownian motion. If $X$ is killed on $\partial E$ we notice that $\mathbf{E}_x[\zeta] < \infty$. We underline the fact that if $X$ is reflected on $\partial E \setminus \partial B$ and killed on $\partial B$ with $B \subset E$, we have that $\mathbf{E}_x[\zeta] < \infty$ only if $E$ is non-trap for $X$. 

Apart from smooth domains, examples of non-trap domains are given by snowflakes or more in general scale irregular fractals as in figures \ref{fig1} and \ref{fig2} (see \cite{CAP} for details). Figures \ref{fig1} and \ref{fig2} are realizations of the random domain obtained by choosing randomly the contraction factor step by step in the construction of the pre-fractal.

\begin{figure}
\includegraphics[scale=.8]{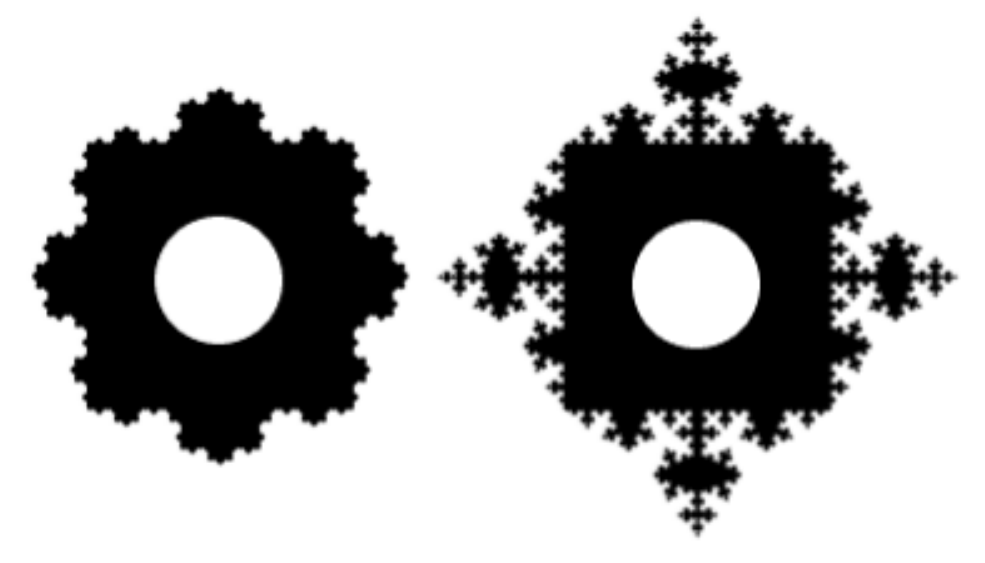} 
\caption{Koch curves outside the square. We have Neumann condition 
on $\partial E \setminus \partial B$ and Dirichlet condition on $\partial B$ where $B$ is the Ball inside $E$. The domains are non-trap for the Brownian motion. 
}
\label{fig1}
\end{figure}

\begin{figure}
\includegraphics[scale=.5]{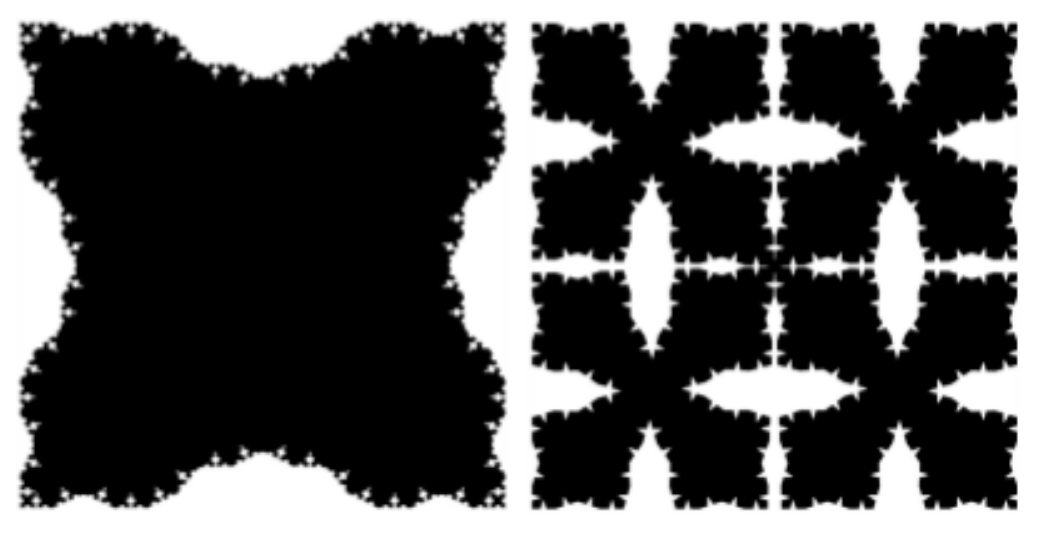} 
\caption{Koch curves inside the square. We have Dirichlet condition on the boundary $\partial E$. The domains are non-trap for the Brownian motion. 
}
\label{fig2}
\end{figure}

\end{remark}


We immediately see that, if $X^T$ is a delayed process,
\begin{align*}
\mathbf{P}_x(\zeta >t) < \mathbf{P}_x(\zeta^T > t) \leq \frac{1}{t} \mathbf{E}_x[\zeta^T],
\end{align*}
whereas, if $X_T$ is a rushed process,
\begin{align*}
\mathbf{P}_x(\zeta^T>t) < \mathbf{P}_x(\zeta > t) \leq \frac{1}{t} \mathbf{E}_x[\zeta].
\end{align*}
%
%
%
%

\begin{remark}
We observe that, from Theorem \ref{B}, 
\begin{align*}
\mathbf{E}_x[(\zeta^L - \zeta)] = \mathbf{E}_x[(H_\zeta - \zeta)] = \int_0^\infty (\mathbf{E}_0[H_s] - s) \mathbf{P}_x(\zeta \in ds) = (\Phi^\prime (0) - 1) \mathbf{E}_x[\zeta] 
\end{align*}
(which is negative/positive depending on $\Phi^\prime$) whereas, from Theorem \ref{meanHzeta},
\begin{align*}
\mathbf{E}_x[(\zeta^H - \zeta)] = & \mathbf{E}_x[(L_\zeta - \zeta)] = \int_0^\infty (\mathbf{E}_0[L_s] - s) \mathbf{P}_x(\zeta \in ds).
\end{align*}
There are explicit representations of $\mathbf{E}_0[L_s]$ only in some case.
\end{remark}


Some interesting consideration have been given in \cite{GDLS} where analytical and numerical results on exit times for stable processes have been investigated. The authors observed that the mean exit time depends on the domain size and the index $\alpha$. Some comparison has been given in terms of $\alpha_1, \alpha_2 \in [0,1]$. They argued that the mean exit time can help us to establish the dominant component between the jump frequency or the the jump size.

\section{Examples and Applications}

In this section we
consider some special examples for the time change $T$ which are of interest in our analysis. We write $f\sim g$ if $f(z)/g(z) \to 1$ as $z \to \infty$. We denote by $C$ any positive constant appearing in the following examples. 

\begin{example}
Let us consider $\Phi(\lambda)=\lambda^\alpha$ with $\alpha \in (0,1)$. 
\begin{itemize}
\item[i)] By using the Laplace transform \eqref{Lap-mittag-L-F} of the density of $L$, we have that
\begin{align*}
\sum_{k \geq 0} \frac{(-\lambda)^k}{k!} \mathbf{E}_0[(L_t)^k] = \mathbf{E}_0[e^{-\lambda t}]= E_\alpha (-\lambda t^\alpha) = \sum_{k\geq 0} \frac{(- \lambda t^\alpha)^k}{\Gamma(\alpha k + 1)}
\end{align*}
from which
\begin{align*}
\mathbf{E}_0[L_t] = \frac{ t^\alpha}{\Gamma(\alpha + 1)}
\end{align*}
where we denoted by $\Gamma(z) = \int_0^\infty e^{-s} s^{z-1} ds$ the gamma function. Thus, for the standard Brownian motion $X$ on $\mathbb{R}^d$ we have that $X^L$ is a diffusion with
\begin{align*}
\mathbf{E}_x[(X\circ L_t)^2] = \mathbf{E}_0[L_t] = \frac{ t^\alpha}{\Gamma(\alpha + 1)}.
\end{align*}
\item[ii)] Let us consider the killed Brownian motion $X$ in $E \subset \mathbb{R}^d$ with the first exit time 
\begin{align*}
\tau_E = \inf \{ t> 0\,:\, X_t \not \in E\}
\end{align*}
such that $\mathbf{E}_x[\tau_E] < \infty$. From Theorem \ref{B}
\begin{align*}
\infty = \mathbf{E}_x[\zeta^L] > \mathbf{E}_x[\zeta] = \mathbf{E}_x[\tau_E]
\end{align*}
and we say that $X$ in $E$ is delayed by $L$. The asymptotic behaviour of the delayed Brownian motion have been investigated also in the interesting paper \cite{Mschilling}.
\end{itemize}

The time-changed process $X_{L_t}$, $t\geq 0$ is usually termed delayed process in the sense that the new clock $L_t$ is a process whose trajectories have plateaus. Actually, $X$ is delayed by $L$ in the sense of Definition \ref{defDelRus}. 
\end{example}

\begin{example}
\label{commentsZV}
Now we focus on the stable subordinator $H$ with symbol $\Phi(\lambda)=\lambda^\alpha$, $\alpha \in (0,1)$. Let $\mathcal{B}_r= \{x \in \mathbb{R}^d\,:\, |x| \leq r\}$ be the ball of radius $r>0$. Let $X$ be a Brownian motion on $\mathbb{R}^d$, $d\geq 1$.
\begin{itemize}
\item[i)] First we recall that (see \cite[formula (4.7)]{dov2010} for example)
\begin{align}
\mathbf{E}_0 [(H_t)^\gamma] = \frac{\Gamma(1- \gamma/\alpha)}{\Gamma(1-\gamma)} t^{\gamma/\alpha}, \quad -\infty < \gamma < \alpha
\label{momH}
\end{align}
Then,
\begin{align*}
\mathbf{E}_x[(X \circ H_t)^2] = \mathbf{E}_0[H_t]
\end{align*}
which is infinite. We should say that $X^H$ behaves like a superdiffusion.
\item[ii)] Let us consider symmetric stable process of order $2\alpha$, then the mean exit time from the ball $\mathcal{B}_r$ is given by (see \cite[formula (24)]{BGR})
\begin{align}
\mathbf{E}_x[\tau_{\mathcal{B}_r}] = \frac{1}{2^{2\alpha}} \frac{1}{\Gamma(\alpha + 1)} \frac{\Gamma(d/2)}{\Gamma(\alpha + d/2)} \left( r^2 - |x|^2 \right)^\alpha.
\label{hittingSS}
\end{align}
Let us consider a Brownian motion $X$ in $\mathbb{R}^d$ with $X_0=x \in \mathcal{B}_r$ killed on $\partial \mathcal{B}_r$. Then $X^H$ is identical in law to a symmetry stable process of order $2\alpha$. That is, $X^H$ is a subordinated killed Brownian motion (see the interesting paper \cite{KSvond}). Let $\zeta^H$ be the lifetime of $X^H$. We have that $\tau_{\mathcal{B}_r} = \zeta^H$ almost surely. 
Let $\zeta$ be the lifetime of the base process $X$. 

For $X_0=0$, that is, $X$ is a Brownian motion on $\mathcal{B}_r$ started at zero, formula \eqref{hittingSS} says that
\begin{align}
\label{rmkZV}
 \mathbf{E}_0[\zeta^H] < \mathbf{E}_0[\zeta]
\end{align}
only if
$$r > 2 \left( \frac{1}{\Gamma(\alpha + 1)} \frac{\Gamma(1+d/2)}{\Gamma(\alpha + d/2)} \right)^{1/(2-2\alpha)} = \rho(\alpha, d)= \rho > 0. $$
In particular, if $d=1$, then formula \eqref{rmkZV} holds for $r  > e^{\frac{3}{2}- \gamma} \approx 2.52$ $\forall \, \alpha$ where $\gamma \approx 0.577$ is the Euler-Mascheroni constant. If $d=2$, then formula \eqref{rmkZV}  holds for $r> 2 e^{1-\gamma} \approx 3.05$, $\forall\, \alpha$.

Let us  consider, for $r>\rho$, the annulus $\Sigma=\mathcal{B}_r \setminus \mathcal{B}_{\rho^*}$ and the set $\mathcal{B}_r = \mathcal{B}_{\rho^*} \cup \Sigma$ where $\rho^*=\sqrt{r^2 - \rho^2}$. For the Brownian motion $X$ started at $x \in \mathcal{B}_r$, we obtain
\begin{align*}
\mathbf{E}_x[\zeta^H] < \mathbf{E}_x[\zeta] \quad \forall\, x \in \mathcal{B}_{\rho^*}
\end{align*}
and
\begin{align*}
\mathbf{E}_x[\zeta^H] > \mathbf{E}_x[\zeta] \quad \forall\, x\in \mathcal{B}_r \setminus \mathcal{B}_{\rho^*} 
\end{align*}
that is, $X$ on $\mathcal{B}_{\rho^*}$ is rushed by $H$ and $X$ on $\Sigma$ is delayed by $H$. Notice that, as $r\to \infty$, $X$ on $\mathbb{R}^d$ is rushed by $H$.\\


Let us consider now the ball $\mathcal{B}_r$ with $r<\rho$. For the Brownian motion $X$ started at $x \in \mathcal{B}_r$ we have that
\begin{align*}
\mathbf{E}_x[\zeta^H] > \mathbf{E}_x[\zeta] \quad \forall\, x \in \mathcal{B}_r
\end{align*}
and $X$ on $\mathcal{B}_r$ is delayed by $H$.

\begin{figure}
\includegraphics[scale=.6]{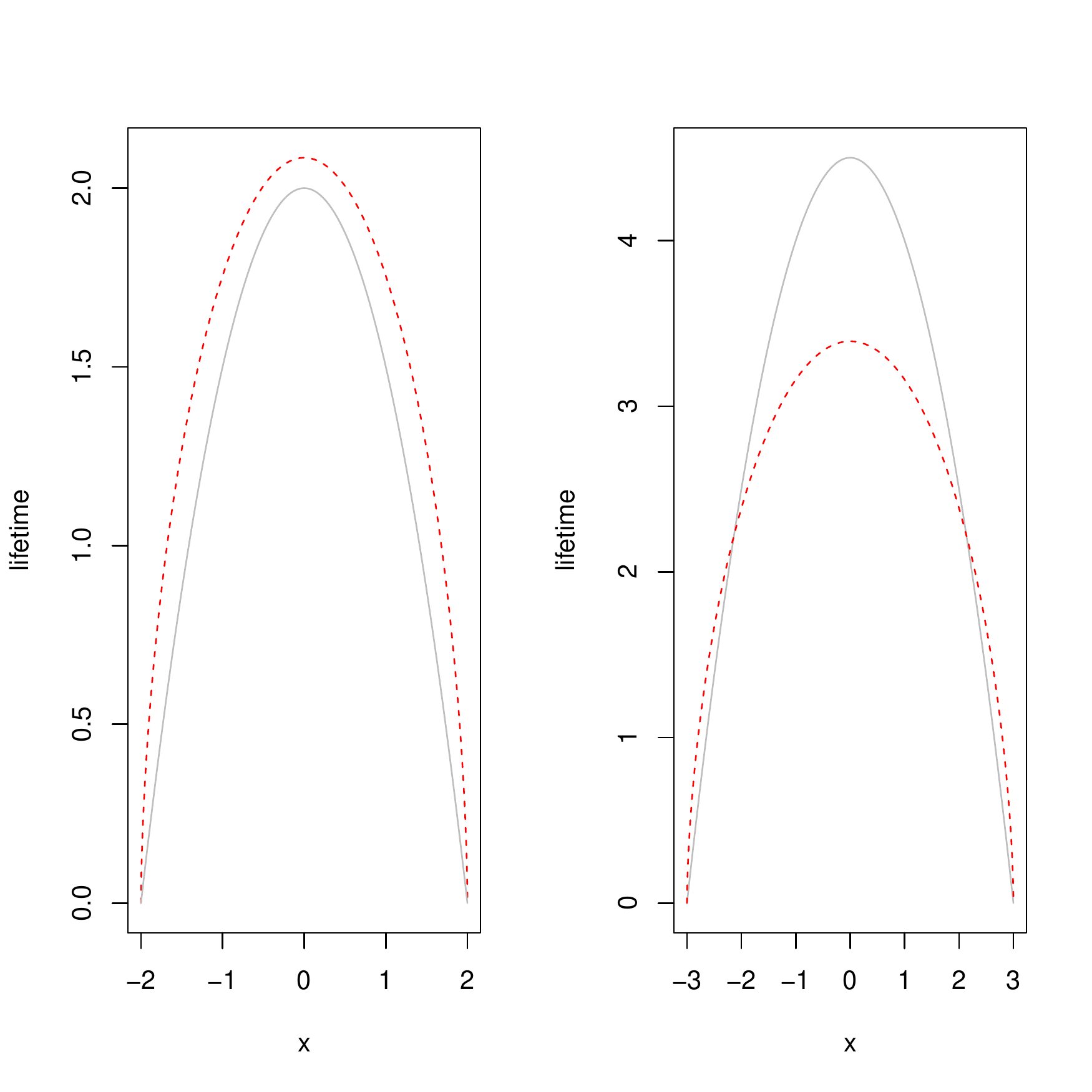} 
\caption{Mean exit times \eqref{hittingSS} from the set (ball) $\mathcal{B}_r \subset \mathbb{R}$ with $r=1$ (first picture) and $r=3$ (second picture). The dashed line refers to $\alpha=0.6$ whereas the solid line refers to $\alpha=1$ (the Brownian motion case). We have that $\rho=2.101$. In the second picture we have that $\rho^*=  2.1356$ determines the set (annulus) $\Sigma= \mathcal{B}_3 \setminus \mathcal{B}_{\rho^*}$ in which $\mathbf{E}_x[\zeta^H] > \mathbf{E}_x[\zeta]$, $\forall\, x \in \Sigma$.}
\end{figure}
\end{itemize}
\end{example}

\begin{example}
Let us consider $\Phi(\lambda)=\lambda^\alpha$ and the exit time from $(-\infty, a)$ for a Brownian motion $X$ on the real line. 
\begin{itemize}
\item[i)] From \eqref{momH}, the mean square displacement of $X \circ H$ is infinite. Furthermore, we immediately see that, for $\gamma \in (0, \alpha)$, the process $X \circ (H)^\gamma$ is a subdiffusion.
\item[ii)] It is well known that (see \cite[page 30]{BatWay})
\begin{align}
\mathbf{P}_{x}(\tau_{(-\infty, a)} \in dt) = \frac{|a-x|}{\sqrt{4\pi t^3}} e^{- \frac{(a-x)^2}{4t}} dt, \quad x<a
\label{densityHittingBM}
\end{align}
is the density of the exit time from $(-\infty, a)$ of a Brownian motion with generator $\Delta$. Formula \eqref{densityHittingBM} gives the density of a $1/2$-stable subordinator  for which 
\begin{align*}
\mathbf{E}_x[(\tau_{(-\infty, a)})^\alpha] = \frac{\Gamma(1/2 - \alpha)}{2^\alpha\, \Gamma(1/2)} |a -x|^{2\alpha }, \quad \alpha < 1/2
\end{align*}
and is infinite for $\alpha>1/2$ (as stated also in \eqref{momH}). Thus, for $\alpha < 1/2$
\begin{align*}
\mathbf{E}_x[(\tau_{(-\infty, a)})^\alpha] < \mathbf{E}_x[\tau_{(-\infty, a)}] = \infty 
\end{align*} 
This means that for the killed Brownian motion $X_t$ on $E=(-\infty, a)$, from Theorem \ref{meanHzeta}, we obtain for $\alpha<1/2$
$$\mathbf{E}_x[\zeta^H] = \frac{1}{\Gamma(\alpha + 1)} \mathbf{E}_x[(\zeta)^\alpha] < \mathbf{E}_x[\zeta]= \mathbf{E}_x[\tau_E]$$
and we say that $X$ is rushed by $H$ (only if $\alpha < 1/2$).
\end{itemize}
\end{example}

\begin{example}
(Sierpinski gasket) 
Consider $\Phi(\lambda)=\lambda^\alpha$. Let $X$ the Brownian motion on a Sierpinski gasket with walk dimension $d_w= \log(5)/\log(2) \simeq 2.32$. 
\begin{itemize}
\item[i)] With formula \eqref{momH} at hand, from  \cite[Corollary 1.6]{BarPerk}, we know that there exists $C$ such that, $\forall\, t>0$,
\begin{align*}
\mathbf{E}_x[(X \circ H_t)^2] = C\, \mathbf{E}_0[(H_t)^{2/d_w}] = \left\lbrace
\begin{array}{ll}
\displaystyle C\frac{\Gamma(1-2/(\alpha\, d_w))}{\Gamma(1-2/d_w)} t^{2/(\alpha\, d_w)}, & 2/d_w <\alpha\\
\displaystyle \infty, & 2/d_w \geq \alpha
\end{array}
\right. .
\end{align*}
From this we argue that, if $H$ is of order $\alpha > 2/d_w$, the Brownian motion on the gasket time-changed by $H$, $X^H_t$,  maintains the subdiffusive behaviour. On the other hand, the behaviour for $\alpha < 2/d_w$ may be related, roughly, to a superdiffusive behaviour.\\  

We notice that the previous argument applies for $d_w>2$, that is for Brownian motions on various fractal domains. 
\item[ii)] Focus on the random time given by the inverse $L$ of the stable process $H$. We have that
\begin{align*}
\mathbf{E}_x[(X \circ L_t)^2] = C\, \mathbf{E}_0[(L_t)^{2/d_w}] \sim t^{2\alpha / d_w}.
\end{align*} 
Thus, $\forall\, \alpha$, $X^L_t$ exhibits a subdiffusive behaviour. 
\end{itemize}
\end{example}
 
\begin{example} 
\label{exampleGammaInverso}
Let us focus on the symbol $\Phi(\lambda)= a \ln (1+ \lambda/b)$, $a>0$, $b>0$, that is $L$ is an inverse to a gamma subordinator $H$.

\begin{itemize}
\item[i)] It is well-known that, for an inverse process $L$, there exist two positive constants $c_1, c_2$ such that (\cite{BerBook})
\begin{align}
c_1 / \Phi(1/t) <  \mathbf{E}_0[L_t] < c_2 / \Phi(1/t).
\end{align}
Than, for the symbol of a gamma subordinator it holds that 
\begin{align*}
\Phi(1/t) = \frac{1}{t} \ln \left( 1+ \frac{1}{b t} \right)^{a t} \sim  \frac{a}{b} \frac{1}{t} \quad \textrm{as } \, t \to \infty.
\end{align*}
Thus, the inverse $L$ has the property
\begin{align*}
\mathbf{E}_0[L_t]  \sim C \frac{b}{a} t \quad \textrm{as } \, t \to \infty
\end{align*}
with $a,b>0$. We have a process with natural diffusivity, that is
\begin{align*}
\mathbf{E}_x[(X\circ L_t)^2] \sim C\,  t, \quad as\ t \to \infty.
\end{align*}
\item[ii)] By considering the time-changed process $X \circ L_t$ we have that (see Theorem \ref{B})
\begin{align*}
\mathbf{E}_x[\zeta^L] = \frac{a}{b} \mathbf{E}_x[\zeta]
\end{align*}
that is, if $a<b$ then $X$ is rushed by $L$ whereas, if $a>b$ then $X$ is delayed by $L$.
\end{itemize}
 
\end{example}

 \begin{example} 
 \label{exampleGammaH}
Let us consider now the gamma subordinator $H$ introduced in the previous example. 
\begin{itemize}
\item[i)] From
\begin{align}
\mathbf{E}_0[H_t] = \frac{d}{d\lambda} \left(1 - e^{- t \Phi(\lambda)} \right) \bigg|_{\lambda=0} = t \Phi^\prime(0)
\label{phiH}
\end{align} 
we obtain
\begin{align*}
\mathbf{E}_0[H_t] = \frac{a}{b} t
\end{align*}
with $a,b>0$. We still have a process with natural diffusivity, that is 
\begin{align*}
\mathbf{E}_x[(X\circ H_t)^2] \sim C\,  t, \quad as\ t \to \infty.
\end{align*}
\item[ii)] For $X \circ H$ , from Theorem \ref{meanHzeta} we have that
\begin{align}
\mathbf{E}_x[\zeta^H] = \mathbf{E}_x[L_\zeta].
\label{lifeH}
\end{align}
We show that
\begin{align}
\label{tmp-a-coeff}
\mathbf{E}_x[L_\zeta] > \mathbf{E}_x[\zeta]
\end{align}
and the base process $X$ is delayed by the random time $H$ only if $a<b$. For the sake of simplicity we set $b=1$. By using the explicit representation of $\mathbf{E}_0[L_t]$  given in \cite[formula  (21)]{kumarPS} we write
\begin{align*}
\mathbf{E}_x[L_\zeta] = \frac{1}{a}\mathbf{E}_x[\zeta] + \frac{1}{2a} - \frac{1}{a} \mathbf{E}_x\left[ e^{-\zeta} \int_0^\infty \frac{e^{-y\zeta}}{(1+y)\, [(\ln y)^2 + \pi^2]} dy \right].
\end{align*}
By considering the integral representation (\cite[pag 2130]{BergPed}, \cite[pag. 989]{QiZhang})
\begin{align}
\label{formulaLN}
\frac{1}{\ln(1+z)} = \frac{1}{z} + \int_1^\infty \frac{1}{(z+y)[(\ln(y-1))^2 + \pi^2]} dy, \quad z \in \mathbb{C} \setminus (-\infty , 0]
\end{align}
and the fact that $\mathbf{E}_x[e^{-(y+1)\zeta}] < 1$, $\forall\, y$ we get that 
\begin{align*}
\mathbf{E}_x\left[ e^{-\zeta} \int_0^\infty \frac{e^{-y\zeta}}{(1+y)\, [(\ln y)^2 + \pi^2]} dy \right] < \int_0^\infty \frac{1}{(1+y)\, [(\ln y)^2 + \pi^2]} dy.
\end{align*}
From \eqref{formulaLN}, we have 
\begin{align*}
\int_0^\infty \frac{1}{(1+y)\, [(\ln y)^2 + \pi^2]} dy = \lim_{z\to 0} \left( \frac{1}{\ln (1+z)} - \frac{1}{z} \right) = \frac{1}{2}
\end{align*}
and we obtain that
\begin{align*}
\mathbf{E}_x[L_\zeta] > \frac{1}{a} \mathbf{E}_x[\zeta].
\end{align*}
The inequality \eqref{tmp-a-coeff} follows only if $a<1$. Then, we say that $X$ is delayed by $H$. 
\end{itemize}
\end{example}

\begin{example}
\label{ExTSS}
A further important example is given by $\Phi(\lambda) = (\lambda + \eta)^\alpha - \eta^\alpha$ for the relativistic stable subordinator $H$. The associated time operator $\mathfrak{D}^\Phi_t$ is usually called tempered fractional derivative (\cite{Beghin}). Tempered stable distribution is particularly attractive in modelling transition from the initial subdiffusive character of motion to the standard diffusion for long times (for the applications of tempered stable distributions see, for example, \cite{jeon,Stani}). Let us consider the inverse $L$ to the subordinator $H$ and the time-changed process $X \circ L$ where $X$ is the Brownian motion.
\begin{itemize}
\item[i)] From Proposition 3.1 in \cite{KumVell}, for large $t$,
\begin{align*}
\mathbf{E}_x[(X \circ L_t)^2] = \mathbf{E}_0[L_t] \sim \frac{1}{\alpha \eta^{\alpha - 1}} t
\end{align*}
and we have a natural diffusion.
\item[ii)] From Theorem \ref{B} we get that
\begin{align*}
\mathbf{E}_x[\zeta^L] =\alpha \eta^{\alpha -1} \mathbf{E}_x[\zeta]
\end{align*}
that is, if $\alpha \eta^{\alpha-1} < 1$ then the process $X$ is rushed by $L$,  whereas if $\alpha \eta^{\alpha-1} > 1$ then the process $X$ is delayed by $L$.
\end{itemize}
\end{example}

\begin{example}
We now consider a further instructive example which does not completely fit the previous scenarios. More precisely, we deal with the one-dimensional fractional Brownian motion $B^\mathcal{H}$ whose density solves
\begin{align*}
\partial_t u = \mathcal{H} t^{2\mathcal{H}-1} \partial_{xx} u, \quad t >0,\, x \in \mathbb{R}, \;\; \mathcal{H} \in (0,1)
\end{align*}
where $\mathcal{H}$ is the Hurst exponent. Although this process is not in the class introduced in Section 2 we consider such a process because of the second moment $\mathbf{E}_x[(B^\mathcal{H}_t)^2] = t^{2\mathcal{H}}$. We provide the following characterizations as in the previous analysis. 
\begin{itemize}
\item[i)] We have that
\begin{align*}
\mathbf{E}_x[(B^\mathcal{H} \circ H_t)^2] = \mathbf{E}_0[(H_t)^{2\mathcal{H}}], \quad \mathbf{E}_x[(B^\mathcal{H} \circ L_t)^2] = \mathbf{E}_0[(L_t)^{2\mathcal{H}}].
\end{align*}
Let $L$ be an inverse to $H$ with symbol $\Phi(\lambda) = a\ln (1+ \lambda/b)$. Then,
\begin{align*}
\mathbf{P}_0(H_t \in dx)= \frac{b^{at} x^{at -1}}{\Gamma(at)} e^{-bx}dx
\end{align*}
and, for large $t$, 
\begin{align*}
\mathbf{E}_0[(H_t)^\gamma] = \frac{1}{b^\gamma} \frac{\Gamma(at + \gamma)}{\Gamma(at)} \sim \left(\frac{a}{b}t \right)^\gamma.
\end{align*}
Using the relation $\mathbf{P}_0(L_t<x) = \mathbf{P}_0(H_x >t)$ we obtain density and moments for $L_t$. By following  \cite{kumarPS} we are able to obtain
\begin{align*}
\int_0^\infty e^{-\lambda t} \mathbf{E}_0[(L_t)^\gamma] dt =  \frac{\Gamma(1 + \gamma)}{\lambda \left( a \ln (1+ \lambda/b)\right)^\gamma} \sim \frac{b^\gamma}{a^\gamma} \frac{\Gamma(1+\gamma)}{\lambda^{\gamma +1}} \quad \textrm{as}\; \lambda \to 0
\end{align*}
from which we get the $\gamma$-moment of the inverse gamma subordinator
\begin{align*}
\mathbf{E}_0[(L_t)^\gamma] \sim \left( \frac{b}{a}t \right)^\gamma.
\end{align*} 
Therefore, we have that, for large $t$,
\begin{align*}
\mathbf{E}_x[(B^\mathcal{H} \circ L_t)^2] \sim \left( \frac{b}{a}t \right)^{2\mathcal{H}}, \quad \mathbf{E}_x[(B^\mathcal{H} \circ H_t)^2] \sim \left( \frac{a}{b}t\right)^{2\mathcal{H}}
\end{align*}
and the subordination leads to anomalous diffusion as well as the time change by inverse subordinator.
\item[ii)] By generalizing the examples \ref{exampleGammaInverso} and \ref{exampleGammaH} we obtain
\begin{align}
\label{nnnggg}
\mathbf{E}_x[\zeta^L] = \frac{a}{b} \mathbf{E}_x[\zeta]
\end{align}
where $\zeta$ is the lifetime of $B^\mathcal{H}$ and $\zeta^L$ is the lifetime of $B^\mathcal{H} \circ L$. Notice that formula \eqref{NoSemig} still holds if $P_t$ is not a semigroup (that is the case here for $B^\mathcal{H}$). Thus, Theorem \ref{meanLzeta} leads to \eqref{nnnggg}.
\end{itemize}

In particular, for the (time-changed) process $B^\mathcal{H} \circ L$ we can write the following table
\begin{center}
\begin{tabular}{c|c|c|c}
 & sub & super & natural\\
 \hline
rushed & $a<b$, $\mathcal{H} < 1/2$ & $a<b$, $\mathcal{H} > 1/2$ & $\mathcal{H}=1/2$\\
\hline
delayed & $a>b$, $\mathcal{H} < 1/2$ & $a>b$, $\mathcal{H} > 1/2$ &$\mathcal{H}=1/2$
\end{tabular} .
\end{center}
The time change for the fractional Brownian motion has been recently considered in \cite{KGWP}. The authors considered the tempered stable subordinator introduced in Example \ref{ExTSS}.
\end{example}

\textbf{Acknowledgement.} The authors would like to thank Prof. Zoran Vondra\v{c}ek for precious discussions and valuable comments.

\textbf{Grant.} The authors are members of GNAMPA (INdAM) and are partially supported by Grants Ateneo \lq\lq Sapienza" 2017.


\bigskip 

\begin{thebibliography}{00}

%
%

\bibitem{BarPerk}
{M. T. Barlow, E. A. Perkins},
{Brownian motion on the Sierpinski gasket},
{Probability Theory and Related Fields, 79, 4, (1988) 543-623}.


\bibitem{Baz}
{E. G. Bazhlekova,} 
{Subordination principle for fractional evolution equations,} 
{Frac. Calc. Appl. Anal. 3 (2000), 213-230.}



\bibitem{BatWay}
{R. N. Bhattacharya, E. C. Waymire,} 
{Stochastic Processes with Applications, (2009) SIAM Classics in Applied Mathematics Series.}


\bibitem{Beghin}
{L. Beghin,} 
{On fractional tempered stable processes and their governing differential equations,}
{Journal of Computational Physics, 293 (2015), 29-39.}


%
%




\bibitem{BergPed}
{C. Berg, H. L. Pedersen,}
{A one-parameter family of pick functions defined by the gamma function and related to the volume of the unit ball in n-space,} 
{Proceedings of the American Mathematical Society,  139, No. 6, 2011,  2121-2132.}


\bibitem{BerBook} {J. Bertoin,} 
{Subordinators: Examples and Applications.} 
{In: Bernard P. (eds) Lectures on Probability Theory and Statistics. Lecture Notes in Mathematics, vol 1717. Springer, Berlin, Heidelberg, 1999.}

\bibitem{BluGet68} 
{R. M. Blumenthal, R. K. Getoor,}
\emph{Markov Processes and Potential Theory}, 
Academic Press, New York, 1968.


\bibitem{BGR} 
{K. Bogdan, T. Grzywny, M. Ryznar,} 
{Heat kernel estimates for the fractional Laplacian with Dirichlet conditions,}
{Ann. Probab. 38 (2010), 5, 1901-1923. }


\bibitem{BCM} 
{K. Burdzy, Z.-Q. Chen, D. E. Marshall,}
{Traps for reflected Brownian motion,}
{Math. Z. 252, (2006) 103-132.}


\bibitem{CAP} 
{R. Capitanelli, }
{Robin boundary condition on scale irregular fractals,} 
{Commun.  Pure Appl. Anal. 9 (2010) 1221--1234.}

\bibitem{CapDovVIA} 
{R. Capitanelli, M. D'Ovidio,}
{Fractional equations via convergence of forms, }
{Fractional Calculus and Applied Analysis, 22 (2019) 844 - 870.}


\bibitem{CastVulp}
{P. Castiglione,  A. Mazzino, P. Muratore-Ginanneschi, A. Vulpiani,}
{On strong anomalous diffusion,}
{Physica D 134 (1999) 75-93.}

\bibitem{chen} 
{Z.-Q. Chen,} 
{Time fractional equations and probabilistic representation,} 
Chaos, Solitons \& Fractals, 102, (2017)  168-174.

%
\bibitem{Chen-book} 
{Z.-Q. Chen,  M. Fukushima,}
{Symmetric Markov Processes, Time Change, and Boundary Theory,} 
London Mathematical Society Monographs, Princeton University Press, 2012.

%

\bibitem{Clark}
{P. Clark,} 
{A subordinated stochastic process model with finite variance for speculative prices,} {Econometrica 41 (1) (1973) 135-155.}
%
%


\bibitem{DovJMA}
{M. D'Ovidio,}
{Continuous random walks and fractional powers of operators, }
{Journal of Mathematical Analysis and Applications, 411, (2014) 362-371}

\bibitem{dov2010}
{M. D'Ovidio,}
{Explicit solutions to fractional diffusion equations via Generalized Gamma Convolution,}
{Electronic Communications in Probability, 15, (2010), 457-474.}



\bibitem{DOTtelegraph} 
{M. D'Ovidio, B. Toaldo,  E. Orsingher,}
{Time changed processes governed by space-time fractional telegraph equations,}
{Stochastic Analysis and Applications, 32, (2014) 6, 1009-1045.}


\bibitem{FGIS}
{R. Failla, P. Grigolini, M. Ignaccolo, A. Schwettmann,} 
{Random growth of interfaces as a subordinated process,} 
{Phys. Rev. E 70 (2004) 010101.}



\bibitem{Foong}
{S. K. Foong}, 
{First passage time, maximum displacement and Kac's solution of the telegrapher equation}, 
{Phys. Rev. A 46 (1992) R707-R7lO.} 



\bibitem{Fuk-book} 
{M. Fukushima, Y. Oshima, M. Takeda,}
{Dirichlet Forms and Symmetric Markov Processes,}
{Walter de Gruyter \& Co, New York, 1994. }

%
\bibitem{GGPS} 
{X. Gabaix, P. Gopikrishnan, V. Plerou, H. E. Stanley,} 
{A theory of power-law distributions in financial market fluctuations,} 
{Nature 423 (6937) (2003), 267-270.}


\bibitem{GDLS}
{T. Gao, J. Duan, X. Li, R. Song,} 
{Mean Exit Time and Escape Probability for Dynamical Systems Driven by Levy Noises,} 
{SIAM J. Scientific Computing. Vol. 36 (2014) A887--A906.}


\bibitem{GoldCox}
{I. Golding, E. C. Cox,}
{ Physical nature of bacterial cytoplasm,} 
{Phys. Rev. Lett. 96 (2006) 098102.}

\bibitem{HMS}
{H. J. Haubold, A. M. Mathai,  R. K. Saxena,} 
{Mittag-Leffler Functions and Their Applications,}
{ Journal of Applied Mathematics, vol. 2011, Article ID 298628, 51 pages, 2011.}


\bibitem{jeon}
{J.-H. Jeon, V. Tejedor, S. Burov, E. Barkai, C. Selhuber-Unkel, K. Berg-Sorensen, L. Oddershede,  R. Metzler,}
{In vivo anomalous diffusion and weak ergodicity breaking of lipid granules,} 
{Phys. Rev. Lett., 106 (2011) 048103.}

\bibitem{KSvond}
{P. Kim, R.Song, Z.Vondra\v{c}ek,} 
{Potential theory of subordinate killed Brownian motion,} 
{Trans. Amer. Math. Soc. 371 (2019) 3917-3969.}

\bibitem{Koc}
{A.N. Kochubei,} 
{The Cauchy problem for evolution equations of fractional order,} {Differential Equations 25 (1989), 967-974.} 
  
 
\bibitem{Kolok}
{V.N. Kolokoltsov,} 
{Generalized continuous-time random walks, subordination by hitting times, and fractional dynamics,} 
{Theory of Probability and its Applications 53, (2009), 594-609.}


\bibitem{KGWP}
{A. Kumar, J. Gajda, A. Wyloma\'{n}ska, R. Polocza\'{n}ski,}
{Fractional Brownian Motion Delayed by Tempered and Inverse Tempered Stable Subordinators,}
{Methodology and Computing in Applied Probability, 21 (2019) 185 - 202.} 
  
\bibitem{KumVell}
{A. Kumar, P. Vellaisamy,}
{Inverse tempered stable subordinators,}
{Statistics \& Probability Letters 103, (2015) 134-141.}



\bibitem{kumarPS}
{A. Kumar, A. Wyloma\'{n}ska, R. Polocza\'{n}ski, S. Sundar,}
{Fractional Brownian motion time-changed by gamma and inverse gamma process,} 
{Phys. A 468 (2017) 648-667.}



\bibitem{MK}
{M. Kwa\'{s}nicki,}
{Ten equivalent definitions of the fractional Laplace operator,} {Fract. Calc. Appl. Anal., Vol. 20, No 1 (2017) pp. 7-51}
 
 
\bibitem{Mschilling}
{M. Magdziarz, R.L.  Schilling,}
{Asymptotic properties of Brownian motion delayed by inverse subordinators,}
{Proc. Amer. Math. Soc. 143 (2015) no. 10, 4485-4501.}

 
\bibitem{Meer}
{M. Meerschaert, E. Nane, P. Vellaisamy,} 
{Fractional Cauchy problems on bounded domains,} 
{Ann. Probab. 37 (3) (2009) 979-1007.}
%
%

\bibitem{Mosco}
{U. Mosco}
{Energy Functionals on Certain Fractal Structures,}
{Journal of Convex Analysis, Vol. 9 (2002) no. 2, 581-600.}

\bibitem{orsT}
{E. Orsingher}, 
{Probability law, flow function, maximum distribution of wave-governed random motions and their connections with Kirchoff's laws},
{Stochastic Processes and their Applications 34, 1 (1990) 49-66.}


\bibitem{ORs}
{E. Orsingher, L. Beghin,} 
{Fractional diffusion equations and processes with randomly varying time,} 
{Ann. Probab. 37 (2009) 206-249.}



%
%


\bibitem{Phillips}
{R. S. Phillips},
{On the generation of semigroups of linear operators,}
{Pacific J. Math. Volume 2, Number 3 (1952) 343-369.}

\bibitem{QiZhang}
{F. Qi,  X.-J. Zhang,} 
{An integral representation, some inequalities, and complete monotonicity of the Bernoulli numbers of the second kind,} 
{Bull. Korean Math. Soc. 52 (2015) no. 3, 987-998.}

%
%
%

\bibitem{SVrel}
{R. Song, Z. Vondra\v{c}ek,} 
{On the relationship between subordinate killed and killed subordinate process,} 
{Electronic Communications in Probability 13 (2008) 325-336.} 


\bibitem{Svond}
{R. Song, Z. Vondra\v{c}ek,} 
{Potential theory of subordinate killed Brownian motion in a domain,} 
{Probab. Theory Relat. Fields 125 (2003) 578 - 592.}



\bibitem{Stani}
{A. Stanislavsky, K. Weron, A. Weron,} 
{Diffusion and relaxation controlled by tempered $\alpha$-stable processes,} 
{Phys. Rev. E (3) 78 (2008) no. 5, 051106, 6.}

\bibitem{toaldo} 
{B. Toaldo,}
{Convolution-type derivatives, hitting-times of subordinators and 
time-changed $C_\infty$-semigroups,} 
{Potential Analysis, 42 (2015) 115-140.}




\bibitem{Vasq}
{J. L. V\'{a}zquez,}
{The Mathematical Theories of Diffusion: Nonlinear and Fractional Diffusion,}
{Nonlocal and Nonlinear Diffusions and Interactions: New Methods and Directions, Lecture Notes in Mathematics, vol 2186. Springer, Cham 205-278, (2017).}



%
%

\end{thebibliography}
\end{document}